\documentclass{amsart}

\newtheorem{theorem}{Theorem}[section]

\theoremstyle{definition}

\newtheorem{corollary}[theorem]{Corollary}

\theoremstyle{remark}
\newtheorem{remark}[theorem]{Remark}

\numberwithin{equation}{section}

\begin{document}

\title{A note on the completeness of $\mathcal{C}_{c}\left(X,Y\right)$}

\author{Jan Harm van der Walt}
\address{Department of Mathematics and Applied Mathematics, University of
Pretoria, Pretoria 0002, South Africa}
\email{janharm.vanderwalt@up.ac.za}



\date{January 22, 2010}



\begin{abstract}
It is known that there are complete, Hausdorff and regular
  convergence vector spaces $X$ and $Y$ such that
  $\mathcal{L}_{c}\left(X,Y\right)$, the space of continuous
  linear mappings from $X$ into $Y$ equipped with the continuous
  convergence structure, is not complete.  In this paper, we give
  sufficient conditions on a convergence vector space $Y$ such
  that $\mathcal{C}_{c}\left(X,Y\right)$ is complete for any
  convergence space $X$.  In particular, we show that this is true
  for every complete and Hausdorff topological vector space $Y$.
\end{abstract}

\maketitle

\section{Introduction}

It is well known \cite{Binz} that $\mathcal{C}_{c}(X)$, the space
of continuous, scalar-valued functions on a convergence space $X$
equipped with the continuous convergence structure, is a complete
convergence vector space.  An immediate consequence of this fact,
see \cite{Beattie}, is that the continuous dual $\mathcal{L}_{c}X$
of a convergence vector space $X$ is complete.  On the other hand,
Butzmann \cite{Butzmann} gave an example of Hausdorff, regular and
complete convergence vector spaces $X$ and $Y$ such that the
convergence vector space $\mathcal{L}_{c}\left(X,Y\right)$ is not
complete. Here, as is standard in the literature, see for instance
\cite{Beattie}, we denote by $\mathcal{L}(X,Y)$ the vector space
of continuous linear mappings from $X$ into $Y$, and
$\mathcal{L}_{c}(X,Y)$ denotes this space equipped with the
continuous convergence structure.

In this paper, we show that if $Y$ is complete, Hausdorff and
topological, then $\mathcal{C}_{c}(X,Y)$ is complete for every
convergence space $X$.  An immediate consequence is that
$\mathcal{L}_{c}(X,Y)$ is complete whenever $X$ and $Y$ are
convergence vector spaces with $Y$ Hausdorff, complete and
topological. This is essentially known in the locally convex case
\cite{Beattie}, \cite{Binz}.

Indeed, if $Y$ is locally convex, Hausdorff and complete, then $Y$
is isomorphic to $\mathcal{L}_{c}\mathcal{L}_{c}Y$, which is a
closed subspace of $\mathcal{C}_{c}(\mathcal{L}_{c}Y)$.  Thus
$\mathcal{L}_{c}(X,Y)$ is isomorphic to a closed subspace of
$\mathcal{C}_{c}\left(X,\mathcal{C}_{c}(\mathcal{L}_{c}Y)\right)$.
By the Universal Property of the continuous convergence structure,
$\mathcal{C}_{c}\left(X,\mathcal{C}_{c}(\mathcal{L}_{c}Y)\right)$
is isomorphic to $\mathcal{C}_{c}\left(X\times
\mathcal{L}_{c}Y\right)$, which is complete \cite{Binz}.  Hence
$\mathcal{L}_{c}(X,Y)$ is a closed subspace of a complete
convergence vector space, and is therefore complete.

\section{A completeness result}

We now show that the following more result holds.  This result
generalizes \cite[Theorem 3.1.15]{Beattie}.
\begin{theorem}\label{T1}
Let $X$ be a convergence space and $Y$ a Hausdorff, complete
topological vector space.  Then $\mathcal{C}_{c}(X,Y)$ is
complete.
\end{theorem}
\begin{proof}
Let $\Phi$ be a Cauchy filter on $\mathcal{C}_{c}(X,Y)$ so that
\begin{eqnarray}
\begin{array}{ll}
\forall & x\in X\mbox{ :} \\
\forall & \mathcal{F}\in\lambda_{X}(x)\mbox{ :} \\
& \omega_{X,Y}\left(\mathcal{F},\Phi-\Phi\right)\in\lambda_{Y}\left(0\right) \\
\end{array},\label{CFilter}
\end{eqnarray}
where $\omega_{X,Y}:X\times \mathcal{C}(X,Y)\rightarrow Y$ is the
evaluation mapping, defined through
\begin{eqnarray}
\omega_{X,Y}\left(x,f\right) = f\left(x\right).\nonumber
\end{eqnarray}
In particular, upon setting $\mathcal{F}=[x]$ in (\ref{CFilter})
we obtain
\begin{eqnarray}
\Phi\left(x\right)-\Phi\left(x\right) = \Phi\left([x]\right) -
\Phi\left([x]\right) = \omega_{X,Y}\left([x],\Phi-\Phi\right) \in
\lambda_{Y}(0)\nonumber
\end{eqnarray}
for every $x\in X$.  Therefore $\Phi(x)$ is a Cauchy filter in $Y$
for every $x\in X$.  Since $Y$ is complete and Hausdorff, it
follows that
\begin{eqnarray}
\begin{array}{ll}
\forall & x\in X\mbox{ :} \\
\exists ! & x_{\Phi}\in Y\mbox{ :} \\
& \Phi(x)\in \lambda_{Y}\left(x_{\Phi}\right) \\
\end{array}.\nonumber
\end{eqnarray}
Define the mapping $f:X\rightarrow Y$ through
\begin{eqnarray}
f:X\ni x\mapsto x_{\Phi}\in Y.\label{PhiLim}
\end{eqnarray}
We show that $f$ is continuous.  Note that, since $Y$ is
topological, there is a collection $\mathcal{B}$ of closed subsets
of $Y$ such that filter $\mathcal{G} = [\mathcal{B}]$ converges to
$0$ and
\begin{eqnarray}
\begin{array}{ll}
\forall & \mathcal{F}\in\lambda_{Y}(0)\mbox{ :} \\
& \mathcal{G}\subseteq \mathcal{F} \\
\end{array}.\label{Assumption}
\end{eqnarray}
Let $\mathcal{F}$ converge to $x_{0}\in X$.  Without loss of
generality, we may assume that $\mathcal{F}\subseteq [x_{0}]$.
Since the filter $\omega_{X,Y}\left(\mathcal{F},\Phi-\Phi\right)$
converges to $0$ in $Y$, it follows by (\ref{Assumption}) that
$\mathcal{G}\subseteq
\omega_{X,Y}\left(\mathcal{F},\Phi-\Phi\right)$. We therefore have
\begin{eqnarray}
\begin{array}{ll}
\forall & B\in\mathcal{B}\mbox{ :} \\
\exists & A_{B}\in \Phi\mbox{ :} \\
\exists & F_{B}\in\mathcal{F}\mbox{ :} \\
& \omega_{X,Y}\left(F_{B},A_{B}-A_{B}\right)\subseteq B \\
\end{array}\nonumber
\end{eqnarray}
so that
\begin{eqnarray}
\left(A_{B}-A_{B}\right)\left(F_{B}\right) = \left\{
g(x)-h(x)\mbox{ }\begin{array}{|l}
g,h\in A_{B} \\
x\in F_{B} \\
\end{array} \right\}\subseteq B.\nonumber
\end{eqnarray}
In particular,
\begin{eqnarray}
\begin{array}{ll}
\forall & x\in F_{B}\mbox{ :} \\
\forall & g\in A_{B}\mbox{ :} \\
& A_{B}\left(x\right) = \{h(x)\mbox{ : }h\in A_{B}\}\subseteq g\left(x\right) + B \\
\end{array}\label{Eq1}
\end{eqnarray}
Since $f\left(x\right)$ is defined as the limit of $\Phi(x)$ in
$Y$, it follows that $f\left(x\right) \in
a_{Y}\left(A_{B}\left(x\right)\right)$, where $a_{Y}$ denotes the
adherence operator in $Y$.  Since $B$ is closed in $Y$, it follows
from (\ref{Eq1}) that
\begin{eqnarray}
\begin{array}{ll}
\forall & g\in A_{B}\mbox{ :} \\
\forall & x\in F_{B}\mbox{ :} \\
& f\left(x\right)\in g\left(x\right) + B \\
\end{array}.\label{Eq2}
\end{eqnarray}
For every $B\in\mathcal{B}$ and $A\in \Phi$, pick some $g\in
A_{B}\cap A$. Since $g$ is continuous, the filter
$g\left(\mathcal{F}\right)$ converges to $g\left(x_{0}\right)$ in
$Y$.  It now follows from (\ref{Assumption}) that
$\mathcal{G}+g\left(x_{0}\right)\subseteq
g\left(\mathcal{F}\right)$.  Fix $B\in\mathcal{B}$.  Then
\begin{eqnarray}
\begin{array}{ll}
\exists & F_{B,0}\in \mathcal{F}\mbox{ :} \\
& g\left(F_{B,0}\right)\subseteq B + g\left(x_{0}\right) \\
\end{array}\nonumber
\end{eqnarray}
so that (\ref{Eq2}) implies
\begin{eqnarray}
f\left(F_{B}\cap F_{B,0}\right)\subseteq
B+g\left(x_{0}\right)\subseteq \left(A_{B}\cap
A\right)\left(x_{0}\right) + B\subseteq A(x_{0})+B.\nonumber
\end{eqnarray}
Since $B\in\mathcal{B}$ and $A\in\Phi$ were arbitrary, it follows
that $\Phi\left(x_{0}\right) + \mathcal{G}\subseteq
f\left(\mathcal{F}\right)$.  By definition,
$\Phi\left(x_{0}\right)$ converges to $f\left(x_{0}\right)$, and
since $\mathcal{G}$ converges to $0$ it follows that
$f\left(\mathcal{F}\right)$ converges to $f\left(x_{0}\right)$
which shows that $f$ is continuous.\\ \\
Now we show that $\Phi$ converges continuously to $f$.  Choose
$x_{0}\in X$ and $\mathcal{F}\in\lambda_{X}\left(x_{0}\right)$ as
above so that we have
\begin{eqnarray}
\begin{array}{ll}
\forall & B\in\mathcal{B}\mbox{ :} \\
\exists & A_{B}\in\Phi\mbox{ :} \\
\exists & F_{B}\in\mathcal{F}\mbox{ :} \\
& g\in A_{B}, x\in F_{B}\Rightarrow f\left(x\right)\in g\left(x\right) + B \\
\end{array}.\label{Eq3}
\end{eqnarray}
Since $f$ is continuous, we also have
\begin{eqnarray}
\begin{array}{ll}
\forall & B\in\mathcal{B}\mbox{ :} \\
\exists & F_{B,0}\in\mathcal{F}\mbox{ :} \\
& f\left(F_{B,0}\right)\subseteq f\left(x_{0}\right) + B \\
\end{array}.\label{Eq4}
\end{eqnarray}
From (\ref{Eq3}) and (\ref{Eq4}) it follows that
\begin{eqnarray}
\begin{array}{ll}
\forall & x\in F_{B}\cap F_{B,0}\mbox{ :} \\
& A_{B}\left(x\right)\subseteq f\left(x\right) - B\subseteq f\left(x_{0}\right) +B-B. \\
\end{array}.\nonumber
\end{eqnarray}
Therefore
\begin{eqnarray}
\omega_{X,Y}\left(F_{B}\cap F_{B,0},A_{B}\right)\subseteq
f\left(x_{0}\right) +B-B.\nonumber
\end{eqnarray}
Consequently, $[f\left(x_{0}\right)]+\mathcal{G}-\mathcal{G}
\subseteq \omega_{X,Y}\left(\Phi,\mathcal{F}\right)$ so that
$\omega_{X,Y}\left(\Phi,\mathcal{F}\right)$ converges to
$f\left(x_{0}\right)$.  Since $x_{0}\in X$ was chosen arbitrary,
it follows that $\Phi$ converges continuously to $f$.  This
completes the proof.
\end{proof}
\begin{corollary}
If $X$ and $Y$ are convergence vector spaces, $Y$ Hausdorff,
complete and topological, then $\mathcal{L}_{c}\left(X,Y\right)$
is complete.
\end{corollary}
\begin{remark}
It should be noted that the proof of Theorem \ref{T1} given here
cannot be used in the case of a nontopological range space $Y$.
Indeed, the proof depends heavily on the existence of a filter
$\mathcal{G}$, with a basis of closed sets, which converges to $0$
in $Y$ and satisfies
\begin{eqnarray}
\begin{array}{ll}
\forall & \mathcal{F}\in\lambda_{Y}\left(0\right)\mbox{ :} \\
& \mathcal{G}\subseteq \mathcal{F} \\
\end{array}.\nonumber
\end{eqnarray}
Clearly the existence of such a filter implies that $Y$ is
pretopological, and hence topological.
\end{remark}

While the techniques used in the proof of Theorem \ref{T1} does
not apply to nontopological spaces, similar arguments suffice if
$Y$ is replaced with a complete, Hausdorff and commutative
topological group.  In particular, the following is true.
\begin{theorem}
Let $X$ be a convergence space, and $Y$ a complete, Hausdorff
commutative topological group.  Then
$\mathcal{C}_{c}\left(X,Y\right)$ is a complete convergence group.
\end{theorem}
Since the proof is based on almost exactly the same arguments used
to verify Theorem \ref{T1} we do not give it here.

Lastly, we mention that the completeness result for
$\mathcal{L}_{c}(X,Y)$, with $Y$ a complete locally convex space,
or more generally any continuously reflexive convergence vector
space \cite{Beattie}, mentioned earlier, has been used
successfully in infinite dimensional analysis, see for instance
\cite{Bjon} and \cite{NelI}.  Our results may therefore have a
wide range of applicability in analysis on non locally convex
spaces \cite{Bayoumi}.

\end{document}